\RequirePackage{amsmath}
\documentclass[10pt, reqno]{amsart}
\usepackage{amssymb,url, color, mathrsfs}
\usepackage[colorlinks=true, bookmarks=true, pdfstartview=FitH, pagebackref=true, linktocpage=true, linkcolor = magenta, citecolor = blue]{hyperref}
\usepackage[short,nodayofweek]{datetime}
\usepackage{longtable}
\usepackage{indentfirst, tabularx}
\usepackage[table]{xcolor}
\usepackage{float, hhline, colortbl}
\usepackage{cite}
\usepackage{graphicx}
\newtheorem{theorem}{Theorem}[section]
\newtheorem{lemma}[theorem]{Lemma}

\def\R{\mathbb R}

\def\Om{\Omega}
\def\ge{\geq}
\def\de{\delta}
\def\De{\Delta} 

\def\na{\nabla}
\def\pa{\partial}
\def\lt{\left}
\def\rt{\right}

\def\sp{\R^n}

\def\i0i{\int_0^\infty}

\def\leqs{<}
\def\geqs{>}

\numberwithin{equation}{section}

\newcounter{hypo}

\makeatletter
 \@addtoreset{equation}{section}
 \makeatother
 
\title[ Liouville type theorems]{A  Liouville type theorem for  fractional elliptic equation with exponential nonlinearity}
\author[A. T. Duong]{Anh Tuan Duong}
\address{Anh Tuan Duong\\Department of Mathematics, Hanoi National University of Education\\136 Xuan Thuy Street, Cau Giay District, Ha Noi, Viet Nam}
\email{tuanda@hnue.edu.vn}
\author[V.H.Nguyen]{Van Hoang Nguyen}
\address{Van Hoang Nguyen\\Institute of Mathematics, Vietnam Academy of Science and Technology, Ha Noi, Viet Nam}
\email{vanhoang0610@yahoo.com and nvhoang@math.ac.vn}
\subjclass{Primary: 35B53, 35J60; Secondary: 35B35.}
 \keywords{Liouville type theorems, Stable solutions, fractional elliptic equations, exponential nonlinearity}
   \date{\bf \today \, at \currenttime }
\begin{document}

\begin{abstract}In this paper, we are concerned with stable solutions to the fractional elliptic equation
	$$
	(-\Delta)^s    u=e^u\mbox{ in }\mathbb R^{N},
	$$
	where  $(-\Delta)^s$ is the fractional Laplacian with  $0<s<1$.  We establish the nonexistence of stable solutions provided that $N<10s$. This result is optimal when $s\uparrow 1$.
\end{abstract}
\maketitle
\section{Introduction}
In this paper, we are interested in the classification of stable  solutions of the fractional Gelfand problem
\begin{equation} \label{e200218}
(-\Delta)^s    u=e^u \mbox{ in }\mathbb R^N
\end{equation}
where $0<s<1$.  Here, the fractional Laplacian $(-\Delta)^s$ is defined as a nonlocal pseudodifferential operator on the Schwartz space of rapidly decreasing functions by
$$(-\Delta)^su(x)=c_{N,s}\lim_{\varepsilon\to 0}\int_{\mathbb{R}^N\setminus B(x,\varepsilon)}\frac{u(x)-u(\xi)}{|x-\xi|^{N+2s}}d\xi,$$
where $c_{N,s}$ is the normalization constant and 
\[
B(x,\varepsilon) = \{y\in \R^N\, ;\, |y-x|\leqs \varepsilon\}.
\]
This operator  is extended  in the distributional sense to  the space $\mathcal L_s(\mathbb R^N)$  with
$$\mathcal L_s(\mathbb{R}^N)=\left\{u\in L_{\rm loc}^1(\mathbb R^N); \int_{\mathbb R^N}\frac{|u(x)|}{(1+|x|)^{N+2s}}dx<\infty\right\}.$$
In addition,  if  $u\in C^{2\sigma}(\mathbb R^N)\cap \mathcal L_s(\mathbb R^N)$ for some  $\sigma>s$, then $(-\Delta)^su(x)$ is well-defined at every point $x\in\mathbb R^N$ see e.g., \cite{MRS16}. In this paper, we mean a  solution $u$ to \eqref{e200218}  by $u\in C^{2\sigma}(\mathbb R^N)\cap \mathcal L_s(\mathbb R^N)$ for some  $\sigma>s$  which satisfies \eqref{e200218} point-wise.

 In recent years, the classification of stable solutions to elliptic equations involving the Laplace operator has been much studied by many mathematicians. The pioneering work in this direction is due to A. Farina \cite{Far07} where the nonexistence of stable solutions to the Lane-Emden equation was completely established.  After that, there have been  many contributions to the classification of stable solutions to elliptic equations/ systems in various cases of nonlinearities \cite{Far07_1,FG14,CF12,Cow13,Cow16,HZ16,Hu15,Hu18,HHM16,DDF11,DNN17}.
 
  In contrast to the local case $s=1$, there has been only a few works dealing with the classification of stable solutions to elliptic equations involving the fractional Laplacian, see \cite{DDW17} and also \cite{FW16,FW17}. In these articles, the authors classified stable solutions, finite Morse index solutions to the fractional Lane-Emden equations by using the Caffarelli-Silvestre extension \cite{CS07}, some nonlinear integral estimates and the monotonicity  formula to overcome the difficulty caused by the presence of the fractional operator. However, the techniques in \cite{DDW17}  do not seem to be directly applicable to the fractional elliptic equation with exponential nonlinearity (known as the Gelfand nonlinearity). 
  
  Recall that in \cite{Far07_1}, the nonexistence of stable solutions to the Gelfand equation 
  \begin{equation}
\label{eq:Gelfand}-\Delta u=e^u \mbox{ in }\R^N
  \end{equation}
 was proved when $1\leq N\leq 9$. This condition is also shown to be sharp in the sense that when $N\geq 10,$ \eqref{eq:Gelfand} admits a radial stable solution \cite{Far07_1}. The purpose of this paper is then to study whether similar results hold true to the fractional Gelfand  equation \eqref{e200218}. Before stating our main result, we recall that a  solution $u$ of  \eqref{e200218}
 is called stable if  
\begin{equation}\label{e08123}
\int_{\mathbb R^n}e^u\phi^2dx\leq \frac{c_{N,s}}{2}\int_{\mathbb R^n}\int_{\mathbb R^n}\frac{(\phi(x)-\phi(y))^2}{|x-y|^{N+2s}}dxdy,\mbox{ for all }\phi\in C^1_c(\mathbb R^N). 
\end{equation}
Let us state  the main result in this paper.
\begin{theorem}\label{t1}
	For $N<10s$,  there is no stable solution of \eqref{e200218}.
	\end{theorem}
	As mentioned above, the Gelfand equation \eqref{eq:Gelfand} admits a stable solution when $N \geq 10$. So, our Theorem \ref{t1} is optimal in the limit $s\uparrow 1$.
	
To prove Theorem \ref{t1}, we do not use monotonicity formula as in \cite{DDW17} since it  seems not applicable due to the exponential nonlinearity. Instead, we use the Caffarelli-Silvestre extension \cite{CS07}, establish a key estimate  (see Lemma \ref{l1311191}) and develop the idea in \cite{Far07_1}. It is worth noticing that  some difficulties arise due to the presence of the fractional Laplacian, especially in the choice of test functions in $C_c^\infty(\R^N)$ to make some integral estimates for the solution. We overcome this difficulty by showing that in fact we can test the equation \eqref{e200218} by a special function $\eta(x) = (1 + |x|^2)^{-\frac{N+2s}2}$ (which does not belong to $C_c^\infty(\R^N)$). This fact will be done by some technical lemmas in Section 2. Note that the choice of function $\eta$ is a crucial technique in the present paper. In forthcoming paper \cite{TuanHoangLE}, the authors develop this approach to prove some Liouville results for the stable solutions of the equation $(-\Delta)^s u = f(u)$ and of the fractional Lane-Emden system.

Let us close the introduction by recalling a standard tool to study the nonlocal problems due to  Caffarelli and Silvestre \cite{CS07}. This result allows us to reduce a nonlocal problem  in $\mathbb R^N$ to a local problem with a nonlinear Neumann boundary condition on $\mathbb R^{N+1}_+$.

\noindent{\bf Theorem A. }{\it Let $0<s<\sigma<1$ and $u\in C^{2\sigma}(\mathbb R^N)\cap \mathcal L_s(\mathbb R^N)$. For $(x,t)\in\mathbb R^{N+1}_+$, we define
	$$U(x,t)=\int_{\mathbb R^N}P_s(x-z,t)u(z)dz,$$
	where $P_s(x,t)$ is the Poisson kernel 
	$$P_s(x,t)=p(N,s)\frac{t^{2s}}{(|x|^2+t^2)^{\frac{N+2s}{2}}}$$
	and $p(N,s)$ is the normalization constant. Then $U\in C^2(\mathbb R^{N+1}_+)\cap C(\overline{\mathbb R^{N+1}_+})$, $t^{1-2s}\partial_t U\in C(\overline{\mathbb R^{N+1}_+})$ and 
	\begin{equation}\label{e1510191}
	\begin{cases}
	-{\rm div}(t^{1-2s}\nabla U)=0&\mbox{ in }\mathbb R^{N+1}_+\\
	U=u&\mbox{ on }\partial\mathbb R^{N+1}_+\\
	-\lim\limits_{t\to 0}t^{1-2s}\partial _t U=\kappa_s(-\Delta)^su&\mbox{ on }\partial\mathbb R^{N+1}_+
	\end{cases}.
	\end{equation}
	Here $\kappa_s=\frac{\Gamma(1-s)}{2^{2s-1}\Gamma(s)}$  and $\Gamma$ is the usual Gamma function.
}

The rest of this paper is devoted to the proof of Theorem \ref{t1}.
\section{Proof of Theorem \ref{t1} }
In this section, we give the proof of Theorem \ref{t1}.  Let us first fix some terminologies.
In what follows, we denote by $C$ a generic positive constant which may change from line to line. Denote also by 
\begin{equation}\label{e2610191}
\eta(x)=(1+|x|^2)^{-\frac{N+2s}{2}},\quad x\in \R^N
\end{equation}
which  plays the role of test function  in the sequel.
We first give an elementary property of $\eta$.
\begin{lemma}\label{l2110191}
Let $\eta$ be defined in \eqref{e2610191}. There holds
\begin{equation}\label{e21101910}
|(-\Delta)^s\eta(x)|\leq C\eta(x),\quad \text{\rm for all }\, x\in \R^N
\end{equation}
where $C$ depends only on $N$ and $s$.
\end{lemma}
\begin{proof}
Denote by $\rho_m=(1+|x|^2)^{-\frac{m}{2}}$ with $m>N$. Then $\rho_{N+2s}=\eta$. The estimate \eqref{e21101910}  is a direct consequence of the following
\begin{equation}\label{eq:boundrho}
|(-\Delta)^s\rho_m(x)|\leq C\eta(x)\mbox{ for all }x\in\mathbb R^N,
\end{equation}
where $C$ depends only on $m,N$ and $s$. Hence,  it is sufficient to prove \eqref{eq:boundrho}. 

Let $R$ be a positive constant and define
\[
f(x_,h) = \frac{\rho_m(x) -\rho_m(x+h)}{|h|^{N +2s}} \chi_{\{|h| > R\}}+ \frac{\rho_m(x) -\rho_m(x+h)+ \na \rho_m(x) \cdot h}{|h|^{N +2s}}  \chi_{\{|h| \leq R\}},
\]
where $\chi_A$ is the characteristic function of the set $A$. Then, it is not hard to see that
\begin{equation*}
	(-\De)^s \rho_m(x) = c_{N,s}\int_{\mathbb R^N}f(x,h)dh.
\end{equation*}
	
A  straightforward computation gives
	\[
	\pa_{kl}^2 \rho_m(x) = -2m (1+|x|^2)^{-m-1}\de_{kl} + 4m(m+1) (1+ |x|^2)^{-m-2} x_k x_l,
	\]
	which implies
	\[
	\|\na^2 \rho_m(x)\|_{op} \leq 2m(2m+3) (1+|x|^2)^{-m-1},
	\] 
here for a $N\times N$ matrix $A$, $\|A\|_{op}$ denotes its operator norm, i.e. $\|A\|_{op} = \max\{|Ax|/|x|\, ;\, x \not= 0\}$. Hence, if  $|h| \leq R$ then 
	\begin{equation}\label{eq:est1}
	|f(x,h)| = \frac{| \int_0^1 \na^2 \rho_m (x + \theta h) h\cdot h (1-\theta) d\theta |}{|h|^{N+2s}} \leq m(2m+3) |h|^{-N+ 2 -2s}.
	\end{equation}
In addition, when $|h| > R$, there holds
	\begin{equation}\label{eq:est2}
	|f(x,h)| \leq \frac{\rho_m(x) + \rho_m(x +h)}{|h|^{N+2s}} \leq 2 |h|^{-N-2s}.
	\end{equation}
	 We first consider $|x| \leq 1$. We choose $R=2$ and define $f(x,h)$ as above. It results from  \eqref{eq:est1} and \eqref{eq:est2} that 
	\begin{align}\label{eq:inball}
	|(-\Delta)^s \rho_m(x)| \leq & c_{N,s} \int_{\{|h|\leq 2\}} |f(x,h)| dh + c_{N,s} \int_{\{|h|> 2\}} |f(x,h)| dh\notag\\
	&\leq m(2m+3) c_{N,s} \int_{\{|h|\leq 2\}} |h|^{-N+2-2s} dh + 2c_{N,s}\int_{\{|h|> 2\}} |h|^{-N-2s} dh\notag \\
	&= c_{N,s} |S^{N-1}| \lt(\frac{m(2m+3) 2^{1-2s}}{1-s}+ \frac{2^{-2s}}{s}\rt).
	\end{align}
	Here $|S^{N-1}|$ is the surface area of the unit sphere $S^{N-1}$ in $\mathbb{R}^N$. 
	
	For $|x| \geq 1$, we take $R= |x|/2$ and define $f(x,h)$ as above. We divide $\R^N$ into three subdomains as follows
	\[
	\Om_1 =\{h\, ;\, |h| < |x|/2\}, \, \Om_2 = \{h\, ;\,|x|/2 \leq |h| < 2|x|\}, \, \Om_3 =\{h\, ;\, |h| \geq 2|x|\}.
	\]
	If  $h\in \Om_1$ then  $|x+ \theta h| \geq |x| -\theta |h| \geq |x|/2$ for any $\theta \in (0,1)$ which implies 
	\[
	\|\na^2 \rho_m(x+\theta h)\|_{op} \leq 2m(2m+3) \Big(1+ \frac{|x|^2}4\Big)^{-m-1} \leq 2^{2m+3}m(2m+3) (1+ |x|^2)^{-m -1}.
	\]
	Hence $|f(x,h)| \leq 2^{2m+2}m(2m+3)(1+ |x|^2)^{-m -1} |h|^{-N+2 -2s}$ on $\Om_1$ which combines with $|x|\geq 1$ gives
	\begin{align}\label{eq:estOm1}
	|\int_{\Om_1} f(x,h) dh| &\leq \int_{\Om_1} |f(x,h)| dh \notag\\
	&\leq \frac{2^{2m-1+2s}}{1-s}m(2m+3) |S^{N-1}| (1+ |x|^2)^{-m -1} |x|^{-2s}\notag\\
	&\leq \frac{2^{2m-1+3s}}{1-s}m(2m+3) |S^{N-1}| (1+ |x|^2)^{-m -1-s}\notag\\
	&\leq \frac{2^{m+\frac{N}2 -2 +3s}}{1-s} m(2m+3) |S^{N-1}| (1+ |x|^2)^{-\frac N2 -s}.
	\end{align}

We next estimate the integral on the second subdomain $\Omega_2$,
	\begin{align}\label{eq:estOm2}
	|\int_{\Om_2} f(x,h) dh| &\leq \int_{\Om_2} |f(x,h)| dh \notag\\
	&\leq 2^{N+2s}|x|^{-N-2s} \lt((1+ |x|^2)^{-m} |x|^N |B_1| + \int_{\R^N} \rho_m dx\rt)\notag\\
	&\leq 2^{\frac32(N+2s)} \lt(|B_1| + \int_{\R^N} \rho_m(x) dx\rt) (1+ |x|^2)^{-\frac N2 -s},
	\end{align}
	where we  have used $|x|\geq 1$,  $\rho_m \in L^1(\R^N)$ and $|B_1|$ denotes the volume of the unit ball in $\mathbb{R}^N$.
	
	If $h\in \Om_3$ then  $|x+ h| \geq |h| - |x| \geq |x|$. Thus $0 \leq f(x,h) \leq (1+ |x|^2)^{-m}$ on  the third subdomain $\Om_3$. Consequently
	\begin{align}\label{eq:estOm3}
	|\int_{\Om_3} f(x,h) dh| &=\int_{\Om_3} f(x,h) dh \notag\\
	&\leq (1+ |x|^2)^{-m} \frac1{2(1-s)} |S^{N-1}| (2 |x|)^{-2s}\notag\\
	&\leq 2^{-m-s + \frac N2} |S^{N-1}| (1+ |x|^2)^{-\frac N2 -s}.
	\end{align}
	Here we have used again $|x|\geq 1$. Finally, the estimate \eqref{eq:boundrho} follows from \eqref{eq:inball}, \eqref{eq:estOm1}, \eqref{eq:estOm2} and \eqref{eq:estOm3}.
\end{proof}
Next, we prove an approximation lemma which enables us to use the function $\eta$ as a test function for the equation \eqref{e200218} (as mentioned in the introduction).
\begin{lemma}\label{Xapxi}
Let $\phi\in C_c^\infty(\mathbb R)$ be a test function satisfying $\phi(t)=1$ if $|t|\leq 1$ and $\phi(t)=0$ if $|t|>2$.  For $R \geqs 1$, we define $\phi_R(x)=\phi\left(\frac{|x|}{R}\right)$. Then there hold
\begin{equation}\label{e2110192}
\lim_{R\to \infty} (-\Delta )^s(\eta\phi_R)(x)\to (-\Delta )^s(\eta)(x)
	\end{equation}
 and 
	\begin{equation}\label{e2110193}
	|(-\Delta )^s(\eta\phi_R)(x)|\leq C(1+R^{-2s})\eta (x),
	\end{equation}
for any $x\in \R^N$. Here,  $C$ is independent of $R$.
\end{lemma}
\begin{proof}
A simple computation gives
\begin{multline}\label{e2110198}
(-\Delta)^s(\eta\phi_R)(x)=\eta(-\Delta)^s\phi_R(x)+\phi_R (-\Delta)^s\eta(x)\\-c_{N,s}\int_{\mathbb R^N}\frac{(\eta(x)-\eta(y))(\phi_R(x)-\phi_R(y))}{|x-y|^{N+2s}}dy.
\end{multline}
It is easy to see that  the first term in the right hand side of \eqref{e2110198} satisfies
\begin{equation}\label{e2110194}
|\eta(x)(-\Delta)^s\phi_R(x)| =R^{-2s}\eta(x) |((-\Delta)^s\phi)(x/R)| \leq CR^{-2s}\eta(x),
\end{equation}
since $\phi \in C_c^\infty(\R^N)$. Concerning the second term, it follows from Lemma \ref{l2110191} that 
\begin{equation}\label{e2110195}
|\phi_R(x)(-\Delta)^s\eta(x)|\leq C\eta(x).
\end{equation}
We now control the integral in the  last term $$I(x):=\int_{\mathbb R^N}\frac{(\eta(x)-\eta(y))(\phi_R(x)-\phi_R(y))}{|x-y|^{N+2s}}dy$$
by dividing into three cases of $x$. 

{\bf Case 1. $|x|\leq \frac{R}{2}$}. 

In this case, we have 
\begin{equation*}
|I(x)|\leq 2\eta(x)\int_{|y|\geq R}\frac{|1-\phi_R(y)|}{|x-y|^{N+2s}}dy,
\end{equation*}
where we have used $\eta(y)\leq \eta(x)$  when $|y|\geq R.$
From a change of variable $x=R\tilde{x}$ and $y=R\tilde{y}$, we obtain 
\begin{equation*}
|I(x)|\leq 2R^{-2s}\eta(x)\int_{1\leq |\tilde{y}|\leq 2}\frac{|1-\phi(|\tilde{y}|)|}{|\tilde{x}-\tilde{y}|^{N+2s}}d\tilde{y}.
\end{equation*}
This and the fact that $|\tilde{x}-\tilde{y}|\geq 1/2$ follow that  
\begin{equation}\label{e2110197}
I(x)\leq CR^{-2s}\eta(x).
\end{equation}

{\bf Case 2. $|x|\geq 3R$.} 

In this case, one has
\begin{equation*}
|I(x)|\leq \int_{|y|\leq 2R}\frac{2\eta(y)\phi_R(y)}{|x-y|^{N+2s}}dy,
\end{equation*}
where we have used $\eta(x)\leq \eta(y)$  for $|y|\leq 2R.$
In addition, $|x-y|\geq \frac{|x|}{3}$. Then, 
$$I(x)\leq \left(\frac{|x|}{3}\right)^{-N-2s}\int_{|y|\leq 2R}|2\eta(y)\phi_R(y)|dy\leq C\eta(x).$$

{\bf Case 3. $\frac{R}{2}<|x|<3R$.} We divide the whole space into three subdomains and decompose
$$I(x)=I_1(x)+I_2(x)+I_3(x), $$
 where
 $$I_1(x)=\int_{|y|\leq \frac{R}{4}}\frac{(\eta(x)-\eta(y))(\phi_R(x)-\phi_R(y))}{|x-y|^{N+2s}}dy,$$
 $$I_2(x)=\int_{|y|\geq 4R}\frac{(\eta(x)-\eta(y))(\phi_R(x)-\phi_R(y))}{|x-y|^{N+2s}}dy$$
 and 
 $$I_3(x)=\int_{\frac{R}{4}< |y|< 4R}\frac{(\eta(x)-\eta(y))(\phi_R(x)-\phi_R(y))}{|x-y|^{N+2s}}dy.$$
By using the fact that $|x-y|\geq C|x|$ when $|y|\leq \frac{R}{4}$ or $|y|\geq 4R$, we obtain as above that 
$$|I_1(x)|+|I_2(x)|\leq C\eta(x).$$
It remains to control $I_3(x)$. First,  by a change of variable $x=R\tilde{x}$ and $y=R\tilde{y}$, one has
$$I_3(x)=R^{-2s}\int_{\frac{1}{4}< |\tilde{y}|< 4}\frac{(\eta(R\tilde{x})-\eta(R\tilde{y}))(\phi(|\tilde{x}|)-\phi(|\tilde{y}|))}{|\tilde{x}-\tilde{y}|^{N+2s}}d\tilde{y}.$$
Note that $\eta$ is a  radial function, $\eta(r) = (1+ r^2)^{-(N+2s)/2}$, 
\[
\eta'(r) =-(N+2s) r (1+ r^2)^{-\frac{N+2s}2 -1},\quad |\eta'(r)| \leq (N+2s) (1+ r^2)^{-\frac{N+2s+1}2}.
\] 
Then, the mean value theorem implies that 
\begin{align*}
|\eta(x)-\eta(y)| &= |\eta(|y|) -\eta(|x|)|\\
&=|\eta'(|x|+\theta(|y|-|x|))|\,||y|-|x||\\
&\leq (N+2s)  (1+(|x|+\theta(|y|-|x|))^2)^{\frac{-N-2s-1}{2}}||y|-|x||
\end{align*}
for some $\theta\in (0,1)$. Obviously, we have in this case that $$|x|+\theta(|y|-|x|) \geq C |x| \geq CR/2$$ for any $\theta \in (0,1)$. Hence, we obtain
$$|\eta(x)-\eta(y)| \leq C R^{-1} \eta(x) |x-y| = C \eta(x) |\tilde x -\tilde y|.$$
Furthermore, it holds
 $$|\phi(|\tilde{x}|)-\phi(|\tilde{y}|)|\leq C|\tilde{x}-\tilde{y}|.$$
 From these estimate, one gets
 \begin{equation*}
 |I_3(x)|\leq CR^{-2s}(1+|x|^2)^{-\frac{N+2s}{2}}\int_{\frac{1}{4}< |\tilde{y}|< 4}\frac{d\tilde{y}}{|\tilde{x}-\tilde{y}|^{N+2s-2}}.
 \end{equation*}
 Recall that $\tilde{x}\in (1/2,3)$ then the integral in the right hand side of this inequality is convergent which yields
 \begin{equation*}
 |I_3(x)|\leq CR^{-2s}\eta(x).
 \end{equation*}
Combining three cases, we obtain 
 \begin{equation}\label{e2110196}
 |I(x)|\leq C(1+R^{-2s})\eta(x),
 \end{equation}
 where $C$ is independent of $R$. 
 Thus, \eqref{e2110193} follows from \eqref{e2110194}, \eqref{e2110195} and \eqref{e2110196}. 
 
 Next, fix $x$ and choose $R$ large enough such that $|x|\leq \frac{R}{2}$. Then, \eqref{e2110194}, \eqref{e2110197} and \eqref{e2110198} imply \eqref{e2110192}.
\end{proof}
The following lemma deals with some integrability.
\begin{lemma}\label{l2110192}
	Let $u$ be a solution of \eqref{e200218}. Then there holds
	\begin{equation}\label{e2110199}
	\int_{\mathbb R^N}e^{u(x)}\eta (x)dx<C\int_{\mathbb R^N}|u(x)|\eta (x)dx,
	\end{equation}
	where $\eta(x)$ is defined in \eqref{e2610191}.
\end{lemma}
\begin{proof}
Let $\phi\in C^\infty_c(\mathbb R)$ be a test function satisfying $\phi(t)=1$ if $|t|\leq 1$ and $\phi(t)=0$ if $|t|>2$.  We define $\phi_R(x)=\phi\left(\frac{|x|}{R}\right)$ for $R \geqs 0$. Multiplying both sides of \eqref{e200218} by $\eta \phi_R \in C_c^\infty(\R^N)$ and using the integration by parts, we get 
	\begin{equation}\label{e2110191}
	\int_{\mathbb R^N}e^u(x)\eta (x)\phi_R(x)dx = \int_{\mathbb R^N}u(x) (-\Delta )^s(\eta\phi_R)(x)dx.
	\end{equation}
By Lebesgue's monotone convergence theorem, it holds
\[
\lim_{R\to \infty} \int_{\mathbb R^N}e^u(x)\eta (x)\phi_R(x)dx = \int_{\R^N} e^{u(x)} \eta(x) dx.
\]
Using Lemma \ref{Xapxi}, the fact $u \in \mathcal L_s(\R^N)$ and  Lebesgue's dominated convergence theorem, we have
\[
\lim_{R\to \infty}\int_{\mathbb R^N}u(x) (-\Delta )^s(\eta\phi_R)(x)dx = \int_{\mathbb R^N}u(x) (-\Delta )^s\eta(x)dx.
\]
So, by letting $R\to \infty$ in \eqref{e2110191}, we obtain
\[
\int_{\R^N} e^{u(x)} \eta(x) dx = 	\int_{\mathbb R^N}u(x) (-\Delta )^s\eta(x)dx.
\]
This equality together with Lemma \ref{e21101910} proves \eqref{e2110199}.
\end{proof}

With Lemma \ref{l2110192} at hand and under stability assumption, we get an uniform integral  estimate of solutions as follows.
 \begin{lemma}\label{l1311191}
 	Let $u$ be a stable solution of \eqref{e200218}. Then for any $0<\alpha<2$, there exists a positive constant $C$ independent of $u$ such that  
 	\begin{equation}\label{e2210195}
 	\int_{\mathbb{R}^N}e^{(2\alpha+1)u(x)}\eta(x)dx\leq C .
 	\end{equation}
 	Here $\eta$ is defined in \eqref{e2610191}.
 \end{lemma}
 \begin{proof}
 	
Suppose that $u$ is a stable solution of \eqref{e200218}. Let $U$ be an extension of $u$ in the sense of Theorem A.  Let $\phi\in C_c^\infty(\mathbb{R})$ be a cutoff function such that $\phi(t)=1$ when $|t|\leq 1$ and $\phi(t)=0$ when $|t|\geq 2$. For $(x,t)\in\mathbb R^{N+1}$, we put 
$$\Phi_R(x,t)=\phi\left(\frac{|(x,t)|}{R}\right) \mbox{ and } \phi_R(x)=\phi\left(\frac{|x|}{R}\right).$$
We  also define
$$\zeta(x,t)=(1+|x|^2+t^2)^{-\frac{N+2s}{4}}$$
which satisfies $\zeta^2(x,0)=\eta(x)$.

The proof of \eqref{e2210195} is quite long and technical. It is then divided into three steps.

\noindent{\bf Step 1.} We prove that 
\begin{equation}\label{e15101911}
\int_{\mathbb R^N}e^{(2\alpha+1)u}\eta\phi_R^2dx\leq C\int_{\mathbb R_+^{N+1}}e^{2\alpha U}|\nabla (\zeta\Phi_R)|^2t^{1-2s}dxdt.
\end{equation}
Indeed, multiplying the first equation in \eqref{e1510191} by $e^{2\alpha U}\zeta\Phi_R^2 $ and integrating by parts, we have
\begin{align}\label{e1510192}
\begin{split}
&\kappa_s\int_{\mathbb R^N}e^{(2\alpha+1)u}\eta\phi_R^2dx=\int_{\mathbb R_+^{N+1}}\nabla U\cdot\nabla (e^{2\alpha U}\zeta^2\Phi_R^2)t^{1-2s}dxdt\\
&=2\alpha \int_{\mathbb R_+^{N+1}}|\nabla U|^2e^{2\alpha U}\zeta^2\Phi_R^2t^{1-2s}dxdt+2\int_{\mathbb R_+^{N+1}}\nabla U\cdot\nabla (\zeta\Phi_R) e^{2\alpha U} \zeta\Phi_R t^{1-2s}dxdt\\
&=\frac{2}{\alpha}\int_{\mathbb R_+^{N+1}}|\nabla (e^{\alpha U})|^2(\zeta\Phi_R)^2t^{1-2s}dxdt+\frac{2}{\alpha}\int_{\mathbb R_+^{N+1}}\nabla (e^{\alpha U})\cdot\nabla (\zeta\Phi_R)\zeta\Phi_R e^{\alpha U}t^{1-2s}dxdt.
\end{split}
\end{align}
Note that 
\begin{multline}\label{e1510193}
\int_{\mathbb R_+^{N+1}}|\nabla (e^{\alpha U}\zeta\Phi_R)|^2t^{1-2s}dxdt=\int_{\mathbb R_+^{N+1}}|\nabla (e^{\alpha U})|^2(\zeta\Phi_R)^2t^{1-2s}dxdt\\
+2\int_{\mathbb R_+^{N+1}}\nabla (e^{\alpha U})\cdot\nabla (\zeta\Phi_R) \zeta\Phi_R e^{\alpha U} t^{1-2s}dxdt+\int_{\mathbb R_+^{N+1}}e^{2\alpha U}|\nabla (\zeta\Phi_R)|^2t^{1-2s}dxdt.
\end{multline}
Inserting \eqref{e1510193} into \eqref{e1510192}, we arrive at
\begin{align}\label{e1510194}
\begin{split}
&\kappa_s\int_{\mathbb R^N}e^{(2\alpha+1)u}\eta \phi_R^2dx=\frac{2}{\alpha}\int_{\mathbb R_+^{N+1}}|\nabla (e^{\alpha U} \zeta\Phi_R)|^2t^{1-2s}dxdt\\
&-\frac{2}{\alpha}\int_{\mathbb R_+^{N+1}}\nabla (e^{\alpha U})\cdot\nabla (\zeta\Phi_R)\zeta\Phi_R e^{\alpha U}t^{1-2s}dxdt.-\frac{2}{\alpha}\int_{\mathbb R_+^{N+1}}e^{2\alpha U}|\nabla (\zeta\Phi_R)|^2t^{1-2s}dxdt\\
&=          \frac{2}{\alpha}\int_{\mathbb R_+^{N+1}}|\nabla (e^{\alpha U}\zeta\Phi_R)|^2t^{1-2s}dxdt-\frac{2}{\alpha}\int_{\mathbb R_+^{N+1}}\nabla (e^{\alpha U}\zeta\Phi_R)\cdot\nabla (\zeta\Phi_R) e^{\alpha U}t^{1-2s}dxdt.
\end{split}
\end{align}
From the Young inequality for $\varepsilon>0$, there holds
\begin{multline}\label{e1510195}
\int_{\mathbb R_+^{N+1}}\nabla (e^{\alpha U}\zeta\Phi_R)\cdot\nabla(\zeta\Phi_R)  e^{\alpha U}t^{1-2s}dxdt\leq \varepsilon\int_{\mathbb R_+^{N+1}}|\nabla (e^{\alpha U}\zeta\Phi_R)|^2t^{1-2s}dxdt\\+\frac{1}{4\varepsilon}\int_{\mathbb R_+^{N+1}}e^{2\alpha U}|\nabla (\zeta\Phi_R)|^2t^{1-2s}dxdt.
\end{multline}
It follows from \eqref{e1510195} and \eqref{e1510194} that 
\begin{multline}\label{e1510196}
\kappa_s\int_{\mathbb R^N}e^{(2\alpha+1)u}\eta\phi_R^2dx\geq \frac{2}{\alpha}(1-\varepsilon)\int_{\mathbb R_+^{N+1}}|\nabla (e^{\alpha U}\zeta\Phi_R)|^2t^{1-2s}dxdt\\-\frac{1}{2\alpha\varepsilon}\int_{\mathbb R_+^{N+1}}e^{2\alpha U}|\nabla (\zeta\Phi_R)|^2t^{1-2s}dxdt.
\end{multline}
Next, we use the stability inequality with the test function $e^{2\alpha u}\phi_R$ to obtain 
\begin{equation}\label{e1510197}
\kappa_s\int_{\mathbb R^N}e^{(2\alpha+1)u}\eta\phi_R^2dx\leq \kappa_s\|e^{\alpha u}\eta\phi_R\|_{H^s(\mathbb R^N)}\leq \int_{\mathbb R_+^{N+1}}|\nabla (e^{\alpha U}\zeta\Phi_R)|^2t^{1-2s}dxdt,
\end{equation}
since $e^{\alpha U}\zeta\Phi_R$ has trace $e^{\alpha u}\eta\phi_R$ on $\pa \R^{N+1}_+$. Hence, \eqref{e1510196} and \eqref{e1510197} yield
\begin{multline}\label{e15101910}
\left(\frac{2}{\alpha}(1-\varepsilon)-1\right)\int_{\mathbb R_+^{N+1}}|\nabla (e^{\alpha U}\zeta\Phi_R)|^2t^{1-2s}dxdt\\\leq \frac{1}{2\alpha\varepsilon}\int_{\mathbb R_+^{N+1}}e^{2\alpha U}|\nabla (\zeta\Phi_R)|^2t^{1-2s}dxdt.
\end{multline}
The constant $\varepsilon$ is taken small enough such that $\frac{2}{\alpha}(1-\varepsilon)-1>0$. Then \eqref{e1510197} and \eqref{e15101910} give \eqref{e15101911}. 

\noindent{\bf Step 2.}
We proceed by estimating the right hand side of \eqref{e15101911} as follows
\begin{equation}\label{e1311191}
\int_{\mathbb R_+^{N+1}}e^{2\alpha U}|\nabla (\zeta\Phi_R)|^2t^{1-2s}dxdt\leq C\int_{\mathbb R^N}e^{2\alpha u}\eta(x)dx.
\end{equation}
Indeed,  the Jensen inequality implies that 
\begin{equation*}
e^{2\alpha U(x,t)}\leq \int_{\mathbb R^N}P_s(x-y,t)e^{2\alpha u(y)}dy.
\end{equation*}
Consequently,
\begin{align}\label{e1510198}
\begin{split}
&I:=\int_{\mathbb R_+^{N+1}}e^{2\alpha U}|\nabla (\zeta\Phi_R)|^2t^{1-2s}dxdt\\
&\leq \int_0^\infty\int_{\mathbb R^N}\left(\int_{\mathbb R^N}P_s(x-y,t)e^{2\alpha u(y)}dy\right)|\nabla (\zeta\Phi_R)|^2t^{1-2s}dxdt\\
&=\int_{\mathbb R^N} e^{2\alpha u(y)}\left(\int_0^\infty \int_{\mathbb R^N}P_s(x-y,t)|\nabla (\zeta\Phi_R)|^2t^{1-2s} dxdt\right) dy,
\end{split}
\end{align}
here, we use the Fubini theorem in the last equality. Let us put 
\begin{multline}\label{e1510199}
\rho_R(y)=\int_0^\infty \int_{\mathbb R^N}P_s(x-y,t)|\nabla (\zeta\Phi_R)|^2t^{1-2s} dxdt\\=p(N,s)\int_0^\infty \int_{\mathbb R^N}\frac{t|\nabla (\zeta\Phi_R)|^2}{(|x-y|^2+t^2)^{\frac{N+2s}{2}}} dxdt.
\end{multline}
Recall that $\Phi_R\in C^\infty_c(\mathbb R^{N+1}_+)$. Then,  $\rho_R$ is continuous on $\mathbb R^N$ and $\rho_R(y)>0$ for all $y\in\mathbb R^N$.
By the Young inequality, there holds
\begin{multline}\label{e21101911}
\rho_R(y)\leq C\int_0^\infty \int_{\mathbb R^N}\frac{t|\nabla \zeta|^2\Phi_R^2}{(|x-y|^2+t^2)^{\frac{N+2s}{2}}} dxdt\\+C\int_0^\infty \int_{\mathbb R^N}\frac{t \zeta^2|\nabla\Phi_R|^2}{(|x-y|^2+t^2)^{\frac{N+2s}{2}}} dxdt.
\end{multline}
We estimate the first term as follows.
\begin{multline*}
J_1(y):=\int_0^\infty \int_{\mathbb R^N}\frac{t|\nabla \zeta|^2\Phi_R^2}{(|x-y|^2+t^2)^{\frac{N+2s}{2}}} dxdt\\\leq C\int_0^\infty \int_{\mathbb R^N}\frac{t(1+|x|^2+t^2)^{-\frac{N+2s}{2}-1}\Phi_R^2}{(|x-y|^2+t^2)^{\frac{N+2s}{2}}} dxdt\\= C\int_0^\infty \int_{\mathbb R^N}\frac{t(1+|y+h|^2+t^2)^{-\frac{N+2s}{2}-1}\phi^2\left(\frac{|(y+h,t)|}{R}\right)}{(|h|^2+t^2)^{\frac{N+2s}{2}}} dhdt.
\end{multline*}
Let us split 
$$\mathbb R^N=\Omega_1\cup\Omega_2\cup\Omega_3,$$
where 
$$\Omega_1=\left\{h\in\mathbb R^N; |h|\leq \frac{|y|}{2}\right\};\;\Omega_2=\left\{h\in\mathbb R^N;  \frac{|y|}{2}<|h|<2|y|\right\}$$
and 
$$\Omega_3=\{h\in\mathbb R^N; |h|>2|y|\}.$$
In $\Omega_1,$ one has $|y+h|\geq \frac{|y|}{2}$. Then,
\begin{align*}
&\int_0^\infty \int_{\Omega_1}\frac{t(1+|y+h|^2+t^2)^{-\frac{N+2s}{2}-1}\phi^2\left(\frac{|(y+h,t)|}{R}\right)}{(|h|^2+t^2)^{\frac{N+2s}{2}}} dhdt\\
&\leq C \eta (y)\int_0^\infty \int_{\Omega_1}\frac{t(1+t^2)^{-1}}{(|h|^2+t^2)^{\frac{N+2s}{2}}} dhdt\\
&\leq C\eta(y)\int_0^\infty t^{1-2s}(1+t^2)^{-1}dt\\
&\leq C\eta(y).
\end{align*}
Similarly, in $\Omega_2,$ one has $|y+h|\geq |y|$ and then 
\begin{align*}
\int_0^\infty \int_{\Omega_2}\frac{t(1+|y+h|^2+t^2)^{-\frac{N+2s}{2}-1}\phi^2\left(\frac{|(y+h,t)|}{R}\right)}{(|h|^2+t^2)^{\frac{N+2s}{2}}} dhdt
&\leq C\eta(y).
\end{align*}
In $\Omega_3$, we have  $|y+h|\geq \frac{|y|}{2}$. Then,
\begin{align*}
&\int_0^\infty \int_{\Omega_3}\frac{t(1+|y+h|^2+t^2)^{-\frac{N+2s}{2}-1}\phi^2\left(\frac{|(y+h,t)|}{R}\right)}{(|h|^2+t^2)^{\frac{N+2s}{2}}} dhdt\\
&\leq C \eta (y)\int_0^\infty \int_{\Omega_3}t(1+|y+h|^2+t^2)^{-\frac{N+2s}{2}-1}\phi^2\left(\frac{|(y+h,t)|}{R}\right)dhdt\\
&\leq C\eta(y)\int_{\Omega_3}(1+|y+h|^2)^{-\frac{N+2s}{2}}dh\\
&\leq C\eta(y).
\end{align*}
Hence
\begin{equation}\label{e2210191}
J_1(y)\leq C\eta(y).
\end{equation}
Similarly, we next control the second term in \eqref{e21101911} as follows.
\begin{align*}
J_2(y)&:=\int_0^\infty \int_{\mathbb R^N}\frac{t \zeta^2|\nabla\Phi_R|^2}{(|x-y|^2+t^2)^{\frac{N+2s}{2}}} dxdt\\
&= \frac{1}{R^2}\int_0^\infty \int_{\mathbb R^N}\frac{t(1+|y+h|^2+t^2)^{-\frac{N+2s}{2}}\left|\nabla\phi\left(\frac{|(y+h,t)|}{R}\right)\right|^2}{(|h|^2+t^2)^{\frac{N+2s}{2}}} dhdt\\
&\leq \frac{1}{R^{2s}}\int_0^\infty \int_{\mathbb R^N}\frac{t(1+|y+h|^2+t^2)^{-\frac{N+2s}{2}-1}\left|\nabla\phi\left(\frac{|(y+h,t)|}{R}\right)\right|^2}{(|h|^2+t^2)^{\frac{N+2s}{2}}} dhdt,
\end{align*}
where we have used the fact that $\nabla\phi\left(\frac{|(y,t)|}{R}\right)=0$ when $|(y+h,t)|\geq 2R$ or $|(y+h,t)| \leq R$. It is then sufficient to use the same arguments as above to arrive at 
\begin{equation}\label{e2210192}
J_2(y)\leq CR^{-2s}\eta(y).
\end{equation}
 Combining \eqref{e1510198}, \eqref{e21101911},\eqref{e2210191} and \eqref{e2210192}, we obtain \eqref{e1311191}.

\noindent{\bf Step 3.} {\it End of the proof of Lemma \ref{l1311191}}  

We   deduce from \eqref{e15101911} and \eqref{e1311191} that 
\begin{equation}\label{e2210193}
\int_{\mathbb R^N}e^{(2\alpha+1)u}\eta\phi_R^2dx\leq C\int_{\mathbb R^N}e^{2\alpha u}\eta(x)dx,
\end{equation}
where $C$ does not depend on  $R$ and $u$.  Letting $R\to\infty$ and using Lebesgue's monotone convergence theorem, we receive
\begin{equation}\label{e2210194}
\int_{\mathbb R^N}e^{(2\alpha+1)u}\eta dx\leq C\int_{\mathbb R^N}e^{2\alpha u}\eta(x)dx.
\end{equation}
Notice that the right hand side of \eqref{e2210194} is finite when $\alpha=1$ thanks to Lemma \ref{l2110192} and $u\in \mathcal L_s(\R^N)$. It is then follows from  a standard bootstrap argument that both sides of \eqref{e2210194} are finite for all $0<\alpha<2$.  Finally,  the H\"{o}lder inequality and \eqref{e2210194} imply \eqref{e2210195}.
\end{proof}

We are now ready to prove Theorem \ref{t1}.

\begin{proof}[Proof of Theorem \ref{t1}]

Suppose that the equation \eqref{e200218} has a stable solution $u$. For $R>0$ large, we put $u_R(x)=u(R x)+2s\ln R$ which is also a stable solution to  the equation \eqref{e200218}. Then, \eqref{e2210195} implies that 
\begin{equation*}
\int_{\mathbb R^N}e^{(2\alpha+1)u_R}\eta(x)dx \leq C.
\end{equation*}
Making the change of variables, we get
\[
R^{-N+2(2\alpha+1)s}\int_{\mathbb R^N}e^{(2\alpha+1)u}\eta(x/R)dx  = \int_{\mathbb R^N}e^{(2\alpha+1)u_R}\eta(x)dx,
\]
which implies 
\begin{equation}\label{e2210196}
\int_{\mathbb R^N}e^{(2\alpha+1)u}\eta(x/R)dx\leq CR^{N-2(2\alpha+1)s}.
\end{equation}
By the assumption $N<10s$, we choose $\alpha$ close to $2$ such that $N-2(2\alpha+1)s<0$. Let $R\to\infty$ in \eqref{e2210196} we get $\int_{\R^N} e^{(2 \alpha + 1)u} dx =0$ which is impossible.
\end{proof}
\bibliographystyle{acm}

\end{document}